\documentclass[reqno,oneside,11pt]{amsart}

\usepackage{amssymb,array,enumitem,setspace,color,multirow}
\usepackage{newpxtext,newpxmath} %ALR: replacement for kpfont, require slightly more space to be ok
\usepackage[scr=boondoxo,cal=euler,bb=pazo]{mathalpha}
\usepackage{mathtools} % needed for cases*, dcases, bsmallmatrix, etc...
\usepackage{geometry}

\usepackage{tcolorbox} %warning and cute boxed

\usepackage{subcaption}
\usepackage{microtype} 

\definecolor{rouge}{rgb}{0.85,0.1,0.15}
\definecolor{forestgreen}{rgb}{0.13,0.54,0.13}
\definecolor{vertforet}{RGB}{237, 135,45} %Cadium
\definecolor{bleu}{rgb}{0.1,0.2,0.8}
\definecolor{g-darkgreen}{RGB}{7, 141, 112}
\definecolor{g-green}{RGB}{38,206,170}
\definecolor{g-lightgreen}{RGB}{152,232,193}
\definecolor{g-lightblue}{RGB}{127,173,226}
\definecolor{g-indigo}{RGB}{80, 73, 204}
\definecolor{g-blue}{RGB}{61,26,120}
\definecolor{gq-green}{RGB}{74,129,34}
\definecolor{gq-mauve}{RGB}{181,126,220}

\usepackage[pdfpagemode=UseNone, pdfstartview={XYZ null null null}]{hyperref}
\hypersetup{
	colorlinks=true,
	citecolor=g-darkgreen,
	linkcolor=g-indigo,
	urlcolor=g-indigo
}

\usepackage{thmtools}%ALR: To solve autoref problem with bad reference
\usepackage[noabbrev]{cleveref} % now this is a package to be proud of!¸

\usepackage[hyperref,
giveninits, %ALR: watchout can have some issues in arXiv
style=alphabetic,
backend=biber,
]
{biblatex}
\renewbibmacro{in:}{} %To remove the intextsize=tiny... Macro
\AtBeginBibliography{\footnotesize} %smaller font

\numberwithin{equation}{section}

\usepackage{tikz}
\usetikzlibrary{decorations.pathmorphing,decorations.pathreplacing,patterns}
\usepackage{tikz-cd}
\tikzset{>=stealth}
\tikzcdset{ arrow style=tikz}
\usetikzlibrary{tqft}
\newcommand{\bbone}{\mathbb 1}
\DeclareMathOperator{\Hom}{Hom}

\theoremstyle{plain}
\newtheorem{theorem}{Theorem}[section]
\newtheorem{lemma}[theorem]{Lemma}
\newtheorem{proposition}[theorem]{Proposition}

\newtheorem{corollary}[theorem]{Corollary}
\newtheorem{definition}[theorem]{Definition}

\newtheorem{remark}[theorem]{Remark}

\newcommand{\lax}{\mathsf L}
\newcommand{\oplax}{\mathsf{opL}}
\newcommand{\frob}{\mathsf F}

\newcommand{\cate}[1]{\mathsf{#1}} %The categories font
\newcommand{\catC}{\cate{C}} %favourite monoidal category
\newcommand{\catD}{\cate{D}} %the second one
\newcommand{\Obj}{\operatorname{Obj}}
\newcommand{\funG}{\mathcal G} %a functor F
\newcommand{\funK}{\mathcal K} %a functor F

\newcommand{\id}{\mathrm{id}}

\tikzset{cob/.style={scale=.7,baseline={(current bounding box.center)},every node/.style={font=\scriptsize}, every tqft/.append style={transform shape}}}
\tikzset{framecob/.style={very thick,fill = gray!80, fill opacity = .15}}

\newcommand{\on}[1]{\operatorname{#1}}

\usepackage[english]{babel}
\addto\extrasenglish{%
}

\begin{document}
	
	\title{Diagrammatics for lax and Frobenius monoidal functors and weak morphism classifiers}
	
	%Many authors in alphabetical order
	%ALR
	\author[A Langlois-R\'emillard]{Alexis Langlois-R\'emillard}
	\address[Alexis Langlois-R\'emillard]{Hausdorff Center for Mathematics, Endenicher Allee 60, 53115 Bonn, Germany
	}
	\email{langlois@uni-bonn.de;
		\href{https://alexisl-r.github.io/}{https://alexisl-r.github.io/};  \href{https://orcid.org/0000-0002-5919-8766}{ORCID:0000-0002-5919-8766}}
	%MS
	\author[M Stroi{\'n}ski]{Mateusz Stroi{\'n}ski}
	\address[Matti Stroi{\'n}ski]{Fachbereich Mathematik, Universität Hamburg, Bundesstra\ss{}e 55, D-20146 Hamburg, Germany}
	\email{\raggedright mateusz.stroinski@uni-hamburg.de;
		\href{https://sites.google.com/view/stroinski}{https://sites.google.com/view/stroinski}; \linebreak
		\href{https://orcid.org/0000-0002-5792-4541}{ORCID:0000-0002-5792-4541}
	}

	%ALR: remove when submitting?
	\date{\today}
	
	
	%%%
	%% ABSTRACT
	%%
	\begin{abstract}
		%Given a strict monoidal category, there exist another strict monoidal category, coined the lax envelope, such that its strict monoidal functors are the lax monoidal functors from the original category. Oplax and Frobenius envelopes with the similar properties also can be defined.  We give a diagrammatic cosntruction of lax, oplax and Frobenius envelopes  and we use them to give simple proof of known properties of lax, oplax and Frobenius monoidal functors. Finally, we specialise to specific monoidal category and retrieve interesting.
		%V2:
		The theory of $2$-monads entails that, for a strict monoidal category $\mathsf{C}$,
		there is a strict monoidal category $\lax(\mathsf{C})$ such that strict monoidal functors from $\lax(\mathsf{C})$ are precisely the lax monoidal functors from $\mathsf{C}$. We give an elementary, diagrammatic, construction of $\lax(\mathsf{C})$ and of its variants for oplax and Frobenius lax functors. The diagrams used are analogous to the diagrammatics for lax monoidal functors studied by McCurdy.
		%ALR: normally I'd try to avoid math in the abstract, but I'll be anal about it when it's end time.
	\end{abstract}
	
	\maketitle
	
	%\tableofcontents %ALR: to remove at the end, but I find it useful to navigate during draft phase

	\onehalfspacing %ALR Bigger spacing for easier reading, remove at the end? It works well with pallatino

	\section{Introduction}
	The different notions of weak homomorphisms of monoidal categories, provided by lax, oplax and Frobenius monoidal functors, have in recent years been the subject of increased interest outside of pure category theory, as exemplified by applications in the theory of Hopf monads (see \cite{McCruden02,BLV11}), of tensor categories and their Drinfeld centres~\cite{FLP24proj,FLP24,FLP25,KhL23,JY26}, quantisation and infinitesimal braidings~\cite{PS22}, virtual tangles~\cite{Brochier19}, topological theories generalising TQFTs~\cite{IKO23}, and condensation in modular fusion categories~\cite{Mulevicius24}.
	
	A diagrammatic calculus for such functors has been developed independently on a number of occasions, including \cite{McCurdy12,PS22,Mulevicius24}, building on ideas of \cite{CS99} and \cite{Mellies06}. These accounts vary in the extent of formality and completeness, with \cite{McCurdy12} being particularly extensive. The diagrams for Frobenius monoidal functors naturally extend the familiar diagrammatics for Frobenius algebras, which, in turn, resemble a ``flat'' variant of two-dimensional TQFTs; see \cite{Lau05}.
	
	In this note, we show that, given a (strict) monoidal category $\mathsf{C}$, the diagrams mentioned above can be used to define a new monoidal category $\lax(\mathsf{C})$ (as well as an oplax and a Frobenius lax variant 	thereof), such that lax (resp. oplax, Frobenius lax) monoidal functors out of $\mathsf{C}$ are equivalently the strict monoidal functors out of $\lax(\mathsf{C})$ (resp. $\oplax(\mathsf C)$, $\frob(\mathsf C)$).
	
	This not only immediately verifies the soundness of said diagrammatic calculus, but also provides an explicit, and quite elegant, answer to the question of the existence of a category satisfying the universal property of $\lax(\mathsf{C})$ described above. In the language of category theory, we verify the existence of a {\it lax morphism classifier} (see \cite[Section~2.4]{LackShulman}) for the $2$-monad on $\mathbf{Cat}$ defining monoidal categories.
	Precisely this question was raised by John Baez in~\cite{Baez_nCatCafe_Lax}, and while the affirmative answer to it can be deduced from the general theory of $2$-monads (see \cite[Theorem~3.13]{BKP89}, \cite[Theorem~2.4]{Lack02}), our direct approach is both more elementary and it gives a more complete answer in this particular case, since we do not merely prove the existence of $\lax(\mathsf{C})$, but also describe it explicitly.
	
	%End of history

	%ALR: If including one applications
	%	Furthermore, when specialising $\catC$ to certain specific categories, it is possible to retrieve certain related questions from different fields via the same framework, providing a potential for communication between them. [Write something in scientific paper language to this regard] As an example, we consider $\catC = \mathbf{TL}$, the Temperley--Lieb category in \cref{sec:applications} and relate them to constructions in skein algebras.
	
	%ALR: applications ?
	%\alr{\bf Potential applications? aka giving free creds to cool people}
	%\alr{The framework developed here can also be useful to highlights similarities in other constructions. For example, choosing certain categories and studying them, for example~\cite{KhL23}.}
	
	%ALR this KhL23 paper seems to study a category quite similar to the one we wanted !

	\begin{comment}
	Another goal of this note is to connect various communities who have employed those methods. For example, lax and oplax classifiers appeared in work of Brochier on virtual knots~\cite{Brochier19}. %ALR: and all other re can refer to 

	%ALR: recap of what we do
	The note begins, in \autoref{sec:construction} we introduce the lax, oplax and Frobenius envelopes of a monoidal category, both via generators and relations and via a diagrammatic construction.  We revisit some results using this diagrammatic calculus \autoref{sec:properties}. %We close with applications to specific categories~\cref{sec:applications}.
	\end{comment}

	\section{Construction of classifiers for (Frobenius) lax monoidal functors}\label{sec:construction}
	
	Let $\mathcal{F}: \mathsf{C} \rightarrow \mathsf{D}$ be a functor of strict monoidal categories.
	A {\it lax monoidal structure} on $\mathcal{F}$ consists of a family of morphisms
	$
	\mathscr{m}_{X,Y}: \mathcal{F}(X) \otimes_{\mathsf{D}} \mathcal{F}(Y) \rightarrow \mathcal{F}(X \otimes_{\mathsf{C}} Y) \text{ and } \mathscr{u}: \bbone_{\mathsf{D}} \rightarrow \mathcal{F}(\bbone_{\mathsf{C}})
	$
	natural in $X,Y \in \mathsf{C}$ and such that the following diagrams commute:
	\begin{equation}\label{associativity-lax}
	% https://q.uiver.app/#q=WzAsNCxbMCwwLCJcXG1hdGhjYWx7Rn0oWCkgXFxvdGltZXMgXFxtYXRoY2Fse0Z9KFkpIFxcb3RpbWVzIFxcbWF0aGNhbHtGfShaKSJdLFsxLDAsIlxcbWF0aGNhbHtGfShYIFxcb3RpbWVzIFkpIFxcb3RpbWVzIFxcbWF0aGNhbHtGfShaKSJdLFswLDEsIlxcbWF0aGNhbHtGfShYKSBcXG90aW1lcyBcXG1hdGhjYWx7Rn0oWSBcXG90aW1lcyBaKSJdLFsxLDEsIlxcbWF0aGNhbHtGfShYIFxcb3RpbWVzIFkgXFxvdGltZXMgWikiXSxbMCwxLCJcXG1hdGhzY3J7bX1fe1gsWX1cXG90aW1lcyBcXG1hdGhjYWx7Rn0oWikiXSxbMCwyLCJcXG1hdGhjYWx7Rn0oWClcXG90aW1lcyBcXG1hdGhzY3J7bX1fe1ksWn0iLDJdLFsyLDMsIlxcbWF0aHNjcnttfV97WCxZXFxvdGltZXMgWn0iLDJdLFsxLDMsIlxcbWF0aHNjcnttfV97WCBcXG90aW1lcyBZLFp9Il1d
	\begin{tikzcd}[ampersand replacement=\&, column sep = large]
	{\mathcal{F}(X) \otimes \mathcal{F}(Y) \otimes \mathcal{F}(Z)} \& {\mathcal{F}(X \otimes Y) \otimes \mathcal{F}(Z)} \\
	{\mathcal{F}(X) \otimes \mathcal{F}(Y \otimes Z)} \& {\mathcal{F}(X \otimes Y \otimes Z)}
	\arrow["{\mathscr{m}_{X,Y}\otimes \mathcal{F}(Z)}", from=1-1, to=1-2]
	\arrow["{\mathcal{F}(X)\otimes \mathscr{m}_{Y,Z}}"', from=1-1, to=2-1]
	\arrow["{\mathscr{m}_{X \otimes Y,Z}}", from=1-2, to=2-2]
	\arrow["{\mathscr{m}_{X,Y\otimes Z}}"', from=2-1, to=2-2]
	\end{tikzcd};
	\end{equation}
	\begin{equation}\label{unitality-lax}
	% https://q.uiver.app/#q=WzAsNCxbMCwwLCJcXGJib25lX3tcXG1hdGhzZntEfX0gXFxvdGltZXMgXFxtYXRoY2Fse0Z9KFgpIl0sWzAsMSwiXFxtYXRoY2Fse0Z9KFgpIl0sWzEsMCwiXFxtYXRoY2Fse0Z9KFxcYmJvbmVfe1xcbWF0aHNme0N9fSkgXFxvdGltZXMgXFxtYXRoY2Fse0Z9KFgpIl0sWzEsMSwiXFxtYXRoY2Fse0Z9KFxcYmJvbmVfe1xcbWF0aHNme0N9fSBcXG90aW1lcyBYKSJdLFswLDEsIj0iLDJdLFswLDIsIlxcbWF0aHNjcnt1fSBcXG90aW1lcyBcXG1hdGhjYWx7Rn0oWCkiXSxbMSwzLCI9Il0sWzIsMywiXFxtYXRoc2Nye219X3tcXGJib25lX3tcXG1hdGhzZntDfX0sWH0iXV0=
	\begin{tikzcd}[ampersand replacement=\&, column sep = large]
	{\bbone_{\mathsf{D}} \otimes \mathcal{F}(X)} \& {\mathcal{F}(\bbone_{\mathsf{C}}) \otimes \mathcal{F}(X)} \\
	{\mathcal{F}(X)} \& {\mathcal{F}(\bbone_{\mathsf{C}} \otimes X)}
	\arrow["{\mathscr{u} \otimes \mathcal{F}(X)}", from=1-1, to=1-2]
	\arrow["{=}"', from=1-1, to=2-1]
	\arrow["{\mathscr{m}_{\bbone_{\mathsf{C}},X}}", from=1-2, to=2-2]
	\arrow["{=}", from=2-1, to=2-2]
	\end{tikzcd}
	\text{ and }
	% https://q.uiver.app/#q=WzAsNCxbMCwwLCJcXG1hdGhjYWx7Rn0oWCkgXFxvdGltZXMgXFxiYm9uZV97XFxtYXRoc2Z7RH19Il0sWzAsMSwiXFxtYXRoY2Fse0Z9KFgpIl0sWzEsMCwiXFxtYXRoY2Fse0Z9KFgpIFxcb3RpbWVzIFxcbWF0aGNhbHtGfShcXGJib25lX3tcXG1hdGhzZntDfX0pIl0sWzEsMSwiXFxtYXRoY2Fse0Z9KFggXFxvdGltZXMgXFxiYm9uZV97XFxtYXRoc2Z7Q319KSJdLFswLDEsIj0iLDJdLFswLDIsIlxcbWF0aGNhbHtGfShYKSBcXG90aW1lcyBcXG1hdGhzY3J7dX0iXSxbMSwzLCI9Il0sWzIsMywiXFxtYXRoc2Nye219X3tYLFxcYmJvbmVfe1xcbWF0aHNme0N9fX0iXV0=
\begin{tikzcd}[ampersand replacement=\&,column sep=large]
	{\mathcal{F}(X) \otimes \bbone_{\mathsf{D}}} \& {\mathcal{F}(X) \otimes \mathcal{F}(\bbone_{\mathsf{C}})} \\
	{\mathcal{F}(X)} \& {\mathcal{F}(X \otimes \bbone_{\mathsf{C}})}
	\arrow["{\mathcal{F}(X) \otimes \mathscr{u}}", from=1-1, to=1-2]
	\arrow["{=}"', from=1-1, to=2-1]
	\arrow["{\mathscr{m}_{X,\bbone_{\mathsf{C}}}}", from=1-2, to=2-2]
	\arrow["{=}", from=2-1, to=2-2]
\end{tikzcd}\ .
	\end{equation}
	An {\it oplax monoidal structure} on $\mathcal{F}$ consists of morphisms $\mathscr{c}_{X,Y}: \mathcal{F}(X \otimes_{\mathsf{C}} Y) \rightarrow \mathcal{F}(X) \otimes_{\mathsf{D}} \mathcal{F}(Y)$ and $\mathscr{e}: \mathcal{F}(\bbone_{\mathsf{C}}) \rightarrow \bbone_{\mathsf{D}}$ endowing $\mathcal{F}^{\on{op}}$ with the structure of a lax monoidal functor.
	A {\it Frobenius monoidal structure} on $\mathcal{F}$ consists of a lax monoidal structure $(\mathscr{m}_{X,Y},\mathscr{u})$ on $\mathcal{F}$ and an oplax monoidal structure $(\mathscr{c}_{X,Y}, \mathscr{e})$ on $\mathcal{F}$, such that the following diagrams commute:
	\begin{equation}\label{Frobeniusaxiom}
	% https://q.uiver.app/#q=WzAsNCxbMCwwLCJcXG1hdGhjYWx7Rn0oWCkgXFxvdGltZXMgXFxtYXRoY2Fse0Z9KFkgXFxvdGltZXMgWikiXSxbMSwwLCJcXG1hdGhjYWx7Rn0oWCkgXFxvdGltZXMgXFxtYXRoY2Fse0Z9KFkpIFxcb3RpbWVzIFxcbWF0aGNhbHtGfShaKSJdLFswLDEsIlxcbWF0aGNhbHtGfShYIFxcb3RpbWVzIFkgXFxvdGltZXMgWikiXSxbMSwxLCJcXG1hdGhjYWx7Rn0oWCBcXG90aW1lcyBZKSBcXG90aW1lcyBcXG1hdGhjYWx7Rn0oWikiXSxbMCwxLCJcXG1hdGhjYWx7Rn0oWCkgXFxvdGltZXMgXFxtYXRoc2Nye2N9X3tZLFp9Il0sWzEsMywiXFxtYXRoc2Nye219X3tYLFl9IFxcb3RpbWVzIFxcbWF0aGNhbHtGfShaKSJdLFswLDIsIlxcbWF0aHNjcnttfV97WCxZXFxvdGltZXMgWn0iLDJdLFsyLDMsIlxcbWF0aHNjcntjfV97WFxcb3RpbWVzIFksIFp9IiwyXV0=
	\begin{tikzcd}[ampersand replacement=\&]
	{\mathcal{F}(X) \otimes \mathcal{F}(Y \otimes Z)} \& {\mathcal{F}(X) \otimes \mathcal{F}(Y) \otimes \mathcal{F}(Z)} \\
	{\mathcal{F}(X \otimes Y \otimes Z)} \& {\mathcal{F}(X \otimes Y) \otimes \mathcal{F}(Z)}
	\arrow["{\mathcal{F}(X) \otimes \mathscr{c}_{Y,Z}}", from=1-1, to=1-2]
	\arrow["{\mathscr{m}_{X,Y\otimes Z}}"', from=1-1, to=2-1]
	\arrow["{\mathscr{m}_{X,Y} \otimes \mathcal{F}(Z)}", from=1-2, to=2-2]
	\arrow["{\mathscr{c}_{X\otimes Y, Z}}"', from=2-1, to=2-2]
	\end{tikzcd}
	\quad
	%\text{ and }
	% https://q.uiver.app/#q=WzAsNCxbMCwwLCJcXG1hdGhjYWx7Rn0oWCBcXG90aW1lcyBZKSBcXG90aW1lcyBcXG1hdGhjYWx7Rn0oWikiXSxbMCwxLCJcXG1hdGhjYWx7Rn0oWFxcb3RpbWVzIFkgXFxvdGltZXMgWikiXSxbMSwwLCJcXG1hdGhjYWx7Rn0oWCkgXFxvdGltZXMgXFxtYXRoY2Fse0Z9KFkpIFxcb3RpbWVzIFxcbWF0aGNhbHtGfShaKSJdLFsxLDEsIlxcbWF0aGNhbHtGfShYKSBcXG90aW1lcyBcXG1hdGhjYWx7Rn0oWSBcXG90aW1lcyBaKSJdLFswLDIsIlxcbWF0aHNjcntjfV97WCxZfVxcb3RpbWVzIFxcbWF0aGNhbHtGfShaKSJdLFsyLDMsIlxcbWF0aGNhbHtGfShYKVxcb3RpbWVzIFxcbWF0aHNjcnttfV97WSxafSJdLFswLDEsIlxcbWF0aHNjcnttfV97WFxcb3RpbWVzIFksWn0iLDJdLFsxLDMsIlxcbWF0aHNjcntjfV97WCxZXFxvdGltZXMgWn0iLDJdXQ==
	\begin{tikzcd}[ampersand replacement=\&]
	{\mathcal{F}(X \otimes Y) \otimes \mathcal{F}(Z)} \& {\mathcal{F}(X) \otimes \mathcal{F}(Y) \otimes \mathcal{F}(Z)} \\
	{\mathcal{F}(X\otimes Y \otimes Z)} \& {\mathcal{F}(X) \otimes \mathcal{F}(Y \otimes Z)}
	\arrow["{\mathscr{c}_{X,Y}\otimes \mathcal{F}(Z)}", from=1-1, to=1-2]
	\arrow["{\mathscr{m}_{X\otimes Y,Z}}"', from=1-1, to=2-1]
	\arrow["{\mathcal{F}(X)\otimes \mathscr{m}_{Y,Z}}", from=1-2, to=2-2]
	\arrow["{\mathscr{c}_{X,Y\otimes Z}}"', from=2-1, to=2-2]
	\end{tikzcd}.
	\end{equation}
	\begin{definition}\label{def:laxcatL}
		Let $\catC$ be a strict monoidal category. The strict monoidal category $\lax(\catC)$ is defined as follows:
		\begin{itemize}
			\item[\textbf{Objects}]
			We let  $\Obj(\lax(\catC))$ be the free monoid on
			$\{\underline x\mid x\in \Obj \catC\}$. It consists of finite lists of symbols of the form $\underline{x}$, for $x \in \Obj \catC$, and the monoidal unit is the empty list, $\emptyset$.
			\item[\textbf{Morphisms}] For $x,y,z \in \Obj \catC$, we add generating morphisms $\{ \underline f: \underline x\to  \underline y \mid f\in \Hom(x,y)\}$ and a further generator $\ell_{x,z}: \underline x \otimes \underline{z} \to \underline{x\otimes z}$. Additionally, a generating morphism $\mathscr j:\emptyset \to \underline{\bbone}$. These generators are subject to the following relations:
			\begin{enumerate}[start=0]
				\item $\underline{\on{id}_{x}} = \on{id}_{\underline{x}}$\label{item:LaxR0}\, ;
				\item $\underline g\circ \underline f = \underline{g\circ f}$; \label{item:LaxR1}
				\item $\ell_{y,y'} \circ \underline f\otimes \underline f' = \underline{f\otimes f'}\circ \ell_{x,x'}$ for $\underline f : \underline x \to \underline y$ and $\underline f': \underline x' \to \underline y'$; \label{item:LaxR2}
				\item $\ell_{x\otimes z,w}\circ (\ell_{x,z}\otimes \id_{\underline w}) = \ell_{x,z\otimes w} \circ (\id_{\underline x} \otimes \ell_{z,w})$; \label{item:LaxR3}
				\item $\ell_{x,\bbone} \circ (\id_{\underline x}\otimes \mathscr j) = \id_{\underline{x}}$\, ; \label{item:LaxR4a}
				\item $\ell_{\bbone, x} \circ (\mathscr j\otimes \id_{\underline{x}}) = \id_{\underline{x}}$\, . \label{item:LaxR4b}
			\end{enumerate}
		\end{itemize}
	\end{definition}
	Let $\mathsf{C,D}$ be strict monoidal categories. We denote by $\mathbf{Lax}(\mathsf{C,D})$ the category of lax monoidal functors from $\mathsf{C}$ to $\mathsf{D}$, and by $\mathbf{Strict}(\mathsf{C,D})$ the category of strict monoidal functors. In both cases, the morphisms are monoidal transformations, i.e. transformations $\sigma: \mathcal{F} \Rightarrow \mathcal{G}$ satisfying $\sigma_{X \otimes Y} \circ \mathscr{m}_{X,Y}^{\mathcal{F}} = \mathscr{m}_{X,Y}^{\mathcal G} \circ (\sigma_{X} \otimes \sigma_{Y})$.
	Define $\mathbf{Oplax}(\mathsf{C,D})$ similarly in terms of oplax monoidal functors, and $\mathbf{Frob}(\mathsf{C,D})$ as the category of Frobenius monoidal functors. In this last case, we take as morphisms the transformations that are simultaneously morphisms of lax and oplax monoidal functors.
	\begin{proposition}\label{prop:presentation}
		For strict monoidal categories $\mathsf C, \mathsf D$, 
		there is an isomorphism of categories
		$\mathbf{Lax}(\mathsf{C,D}) \cong \mathbf{Strict}(\lax(\mathsf{C}), \mathsf{D})$.
	\end{proposition}

\begin{proof}
		By definition of $\lax(\mathsf{C})$, a strict monoidal functor $\mathcal{F}:\lax(\mathsf{C})\rightarrow \mathsf{D}$ is determined by a morphism of monoids $\mathcal{F}_{0}: \Obj (\lax(\mathsf{C})) \rightarrow \Obj(\mathsf{D})$, equivalently a function $\overline{\mathcal{F}_{0}}: \Obj(\mathsf{C}) \rightarrow \Obj (\mathsf{D})$, and by assignments of morphisms $\mathcal{F}(\underline{g}): \mathcal{F}(\underline{x}) \rightarrow \mathcal{F}(\underline{y})$ for any $x,y \in \Obj \mathsf{C}$ and any $g \in \on{Hom}_{\mathsf{C}}(x,y)$, as well as $\mathcal{F}(\ell_{x,z}): \mathcal{F}(\underline{x} \otimes \underline{z}) = \mathcal{F}(\underline{x}) \otimes \mathcal{F}(\underline{z}) \rightarrow \mathcal{F}(\underline{x \otimes z})$  	and $\mathcal{F}(\mathscr j): \mathcal{F}(\emptyset) = \bbone_{\mathsf{D}} \rightarrow \mathcal{F}(\bbone_{\mathsf{C}})$.
		
		Prior to comparing axioms, we now observe that the above data coincides with the data required to specify a lax monoidal functor $\overline{\mathcal{F}}: \mathsf{C} \rightarrow \mathsf{D}$, defined by $\overline{\mathcal{F}}_{0} := \overline{\mathcal{F}_{0}}$ and by $\overline{\mathcal{F}}(g) := \mathcal{F}(\underline{g})$ for any morphism $g$ of $\mathsf{C}$, with candidate lax monoidal structure afforded by $\mathscr{m}^{\overline{\mathcal{F}}}_{x,z} := \mathcal{F}(\ell_{x,z})$ and $\mathscr{u}^{\overline{\mathcal{F}}} := \mathcal{F}(\mathscr{j})$. \Cref{item:LaxR0,item:LaxR1} are equivalent to the functoriality of $\overline{\mathcal{F}}$; \cref{item:LaxR2} is equivalent to the naturality of the candidate lax monoidal structure; \cref{item:LaxR3} is equivalent to axiom~\eqref{associativity-lax} for lax monoidal functors, and \cref{item:LaxR4a,item:LaxR4b} are equivalent to the respective unitality axioms \eqref{unitality-lax} for such functors.	The bijection between objects of $\mathbf{Lax}(\mathsf{C,D})$ and $\mathbf{Strict}(\lax(\mathsf{C}), \mathsf{D})$ follows.
		
		Similarly, a monoidal transformation $\sigma: \mathcal{F} \Rightarrow \mathcal{G}$ in $\mathbf{Strict}(\lax(\mathsf{C}), \mathsf{D})$ satisfies $\sigma_{\underline{x}\otimes \underline{z}} = \sigma_{\underline{x}} \otimes \sigma_{\underline{z}}$; hence, it is determined by the components of the form $\sigma_{\underline{x}}$, using which we define a transformation $\overline{\sigma}: \overline{\mathcal{F}} \Rightarrow \overline{\mathcal{G}}$, by setting $\overline{\sigma}_{x} = \sigma_{\underline{x}}$. Monoidality of $\overline{\sigma}$ is equivalent to the naturality square for $\sigma$ commuting for the morphisms $\ell_{x,z}$ and $\mathscr j$, establishing bijections on $\on{Hom}$-sets. It is easy to verify the functoriality of these assignments.
	\end{proof}
	
	\begin{remark}
		Using the language of \cite{IKO23}, we conclude that the topological theories on $\mathsf{C}$ are precisely the TQFTs on $\lax(\mathsf{C})$.
	\end{remark}
	
	\begin{proposition}
		Let $\mathsf{C}$ be a strict monoidal category. Then $\lax(\mathsf{C}^{\on{op}})^{\on{op}}$ is an oplax morphism classifier for $\mathsf{C}$. In other words, for any strict monoidal category $\mathsf{D}$ we find an isomorphism of categories
		$\mathbf{Oplax}(\mathsf{C,D}) \cong \mathbf{Strict}(\lax(\mathsf{C}^{\on{op}})^{\on{op}}, \mathsf{D})$.
	\end{proposition}
	\begin{proof}
		$
		\mathbf{Oplax}(\mathsf{C,D}) \simeq \mathbf{Lax}(\mathsf{C}^{\on{op}},\mathsf{D}^{\on{op}})^{\on{op}} \simeq \mathbf{Strict}(\lax(\mathsf{C}^{\on{op}}),\mathsf{D}^{\on{op}})^{\on{op}} \simeq \mathbf{Strict}(\lax(\mathsf{C}^{\on{op}})^{\on{op}},\mathsf{D}).
		$
	\end{proof}
	
	\begin{corollary}\label{oplaxcopants}
		The oplax classifier $\oplax(\mathsf{C})$ can be presented analogously to $\lax(\mathsf{C})$, involving generators $\mathscr{k}_{x,z}: \underline{x\otimes z} \rightarrow \underline{x} \otimes \underline{z}$ rather than $\ell_{x,z}$ and $\mathscr{q}: \underline{\bbone} \rightarrow \emptyset$ rather than $\mathscr j$.
	\end{corollary}
	The previous corollary is easy to see with diagrams; see~\eqref{eq:oplaxdiag} and~\eqref{eq:oplax_diag_identities}.

	\begin{definition}
		The Frobenius classifier $\frob(\mathsf{C})$ is defined as follows. We let $\Obj(\frob(\catC))$ be the free monoid on
		$\{\underline x\mid x\in \Obj \catC\}$. For all $x,y \in \mathsf{C}$, we add generating morphisms  $\{ \underline f: \underline x\to  \underline y \mid f\in \Hom(x,y)\}$ and further generators $\ell_{x,z}: \underline{x} \otimes \underline{z} \to \underline{x \otimes z}$ and
		$\mathscr{k}_{x,z}: \underline{x\otimes z} \rightarrow \underline{x} \otimes \underline{z}$. Finally, we add generating morphisms $\mathscr j:\emptyset \to \underline{\bbone}$ and $\mathscr{q}: \underline{\bbone} \rightarrow \emptyset$. 
		
		We impose all the relations of \Cref{def:laxcatL}, as well as ``oppositized'' variants of 
		\cref{item:LaxR2,item:LaxR3,item:LaxR4a,item:LaxR4b} for $\mathscr{k}_{x,x'}$ and $\mathscr{q}$, following \Cref{oplaxcopants}. 	Additionally, we impose the relations
		\begin{equation}\label{Frobeniusrelation1}
		\mathscr{k}_{x \otimes z,w} \circ \ell_{x,z\otimes w} = (\ell_{x,z}\otimes \on{id}_{\underline{w}}) \circ (\on{id}_{\underline{x}}\otimes \mathscr{k}_{z,w})
		\end{equation}
		and
		\begin{equation}\label{Frobeniusrelation2}
		\mathscr{k}_{x,z\otimes w} \circ \ell_{x\otimes z, w} = (\on{id}_{\underline{x}}\otimes \ell_{z,w}) \circ (\mathscr{k}_{x,z} \otimes \on{id}_{\underline{w}}).
		\end{equation}
	\end{definition}
	
	\begin{theorem}\label{Frobeniusclassifier}
		For strict monoidal categories $\mathsf C, \mathsf D$, there is an isomorphism of categories
		$\mathbf{Frob}(\mathsf{C,D}) \cong \mathbf{Strict}(\frob(\mathsf{C}), \mathsf{D})$.
	\end{theorem}
	\begin{proof}
		Let $\mathcal{F}: \frob(\mathsf{C}) \rightarrow \mathsf{D}$ be a strict monoidal functor. Similar to the proof of \Cref{prop:presentation}, the assignments $\overline{\mathcal{F}}(x) = \mathcal{F}(\underline{x})$, $\overline{\mathcal{F}}(f) = \mathcal{F}(\underline{f})$ define a functor $\overline{\mathcal{F}}: \mathsf{C} \rightarrow \mathsf{D}$, and the maps $\mathcal{F}(\ell_{x,z})$ and $\mathcal{F}(\mathscr j)$ define a lax monoidal structure on $\overline{\mathcal{F}}$. Following \Cref{oplaxcopants}, the maps $\mathcal{F}(\mathscr{k}_{x,z})$ and $\mathcal{F}(\mathscr{q})$ define an oplax monoidal structure on $\overline{\mathcal{F}}$. The relations~(\ref{Frobeniusrelation1}, \ref{Frobeniusrelation2}) are equivalent to the axiom~\eqref{Frobeniusaxiom} making $\overline{\mathcal{F}}$ a Frobenius monoidal functor. Also the correspondence for natural transformations extends similarly to the proof of \Cref{prop:presentation}.
	\end{proof}
	Let $\mathcal{E}: \mathsf{C} \rightarrow \frob(\mathsf{C})$ be the functor corresponding to $\on{Id}_{\frob(\mathsf{C})}$ under the correspondence of \Cref{Frobeniusclassifier}. The following is very easy to verify.
	\begin{lemma}\label{factorize}
		Given a Frobenius monoidal functor $\mathcal{F}: \mathsf{C}\rightarrow \mathsf{D}$, we have $\mathcal{F} = \overline{\mathcal{F}}\circ \mathcal{E}$.
	\end{lemma}
	
		\section{Diagrammatic interpretation}\label{sec:diagrammatic}
	The generators and relations for the classifiers $\lax(\mathsf{C})$, $\oplax(\mathsf C)$ and $\frob (\mathsf C)$ defined in \Cref{sec:construction}	can be interpreted using diagrams very similar to those of \cite{McCurdy12,PS13,PS22,Mulevicius24}. More precisely, we interpret the classifier by enclosing the string calculus of the monoidal category inside an envelope.

	We now describe the diagrammatics. For all $x,y,z,w\in\mathsf C$,  $f \in \on{Hom}_{\mathsf{C}}(x,y)$, $g \in \on{Hom}_{\mathsf{C}}(y,z)$, $h\in \on{Hom}_{\mathsf C}(z,w)$, we associate
	\begin{equation} \label{eq:standard_envelope}
	\begin{aligned}
	\id_x &\mapsto \quad
	\begin{tikzpicture}[cob]
	\draw[thick,teal] (0,0) node[below]{$\underline x$} --  (0,2) node[above]{$\underline x$};
	\draw[framecob] (-.5,0) -- (.5,0) -- (.5,2) -- (-.5,2) -- cycle;
	\end{tikzpicture} \ ,
	&\quad
	f
	&\mapsto  \begin{tikzpicture}[cob]
	\draw[thick,teal] (0,0) node[below]{$\underline x$} -- (0,1-.25);
	\draw[thick,teal] (0,2) node[above]{$\underline y$} -- (0,1+.25);
	\draw[thick,teal] (0,1) circle (.25) node{ $f$};
	\draw[framecob] (-.5,0) -- (.5,0) -- (.5,2) -- (-.5,2) -- cycle;
	\end{tikzpicture}\ , \quad & \text{with composition respecting} \quad \begin{tikzpicture}[cob]
	\draw[thick,teal] (0,0) node[below]{$\underline x$} -- (0,1-.25);
	\draw[thick,teal] (0,2)  -- (0,1+.25);
	\draw[thick,teal] (0,1) circle (.25) node{$f$};
	\draw[framecob] (-.5,0) -- (.5,0) -- (.5,2) -- (-.5,2) -- cycle;
	\draw[thick,teal] (0,2)  -- (0,1-.25+2);
	\draw[teal] (0,2) node[above right ]{$\underline y$};
	\draw[thick,teal] (0,4) node[above]{$\underline z$} -- (0,1+2.25);
	\draw[thick,teal] (0,1+2) circle (.25) node{$g$};
	\draw[framecob] (-.5,0+2) -- (.5,0+2) -- (.5,2+2) -- (-.5,2+2) -- cycle;
	\end{tikzpicture} 
	\ &= \
	\begin{tikzpicture}[cob]
	\draw[thick,teal] (0,0) node[below]{$\underline x$} -- (0,2-.42);
	\draw[thick,teal] (0,4) node[above]{$\underline z$}  -- (0,2+.42);
	\draw[thick,teal] (0,2) circle (.42) node{ $g\!\circ\!\! f$};
	\draw[framecob] (-.5,0) -- (.5,0) -- (.5,4) -- (-.5,4) -- cycle;
	\end{tikzpicture}\ .
	\end{aligned}
	\end{equation}
	%%%%%
	%lax%
	%%%%%
	The extra morphisms of $\lax(\mathsf{C})$ are given diagrammatically by the following diagrams for $x,z\in \Obj(\mathsf{C})$
	\begin{equation}\label{eq:laxdiag}
	\begin{aligned}
	\ell_{x,z} &\mapsto \begin{tikzpicture}[cob]
	\draw[thick, teal] (-.5,2) node[below] {$\underline{x}$} -- (.33,4) node[above] {$\underline{x}$};
	\draw[thick, teal] (1.5,2) node[below] {$\underline{z}$} -- (.66,4) node[above] { $\underline{z}$};
	\draw[framecob] (-1,2) -- (0,4) -- (1,4) -- (2,2) -- (1,2) .. controls (1,3) and (0,3).. (0,2)--cycle;
	\end{tikzpicture}\ , 
	&&
	\mathscr j &\mapsto 
	\begin{tikzpicture}[cob]
	\draw[framecob] (0,0) .. controls (0,-1) and (1,-1) .. (1,0) -- cycle;
	\end{tikzpicture}\ , 
	& \quad \text{with compatibility} \quad 
	\begin{tikzpicture}[cob]
	%first leg
	\draw[framecob] (-1,0) -- (-1,2) -- (0,2) -- (0,0) -- cycle;
	\draw[thick, teal] (-.5,0) node[below]{$\underline x$} -- (-.5,.75);
	\draw[thick, teal] (-.5,2) -- (-.5,1.25);
	\draw[thick, teal] (-.5,1) node{$f$} circle(.25);
	%second leg
	\draw[framecob] (2,0) -- (2,2) -- (1,2) -- (1,0) -- cycle;
	\draw[thick, teal] (2-.5,0) node[below]{$\underline z$} -- (2-.5,.75);
	\draw[thick, teal] (2-.5,2) -- (2-.5,1.25);
	\draw[thick, teal] (2-.5,1) node{$h$} circle(.25);
	%top
	\draw[thick, teal] (-.5,2) -- (.33,4) node[above] { $\underline{y}$};
	\draw[thick, teal] (1.5,2) -- (.66,4) node[above] { $\underline{w}$};
	\draw[framecob] (-1,2) -- (0,4) -- (1,4) -- (2,2) -- (1,2) .. controls (1,3) and (0,3) ..  (0,2)--cycle;
	\end{tikzpicture}\
	&  =
	\begin{tikzpicture}[cob]
	%top
	\draw[thick, teal] (.25,4) -- (.25,4.8);
	\draw[thick, teal] (.25,6) node[above]{$\underline y$} -- (.25,5.2);
	\draw[thick, teal] (.75,4) -- (.75,4.8);
	\draw[thick, teal] (.75,6) node[above]{$\underline w$}-- (.75,5.2);
	\draw[thick, teal] (.25,5) node{$f$} circle(.2);
	\draw[thick, teal] (.75,5) node{$h$} circle(.2);
	\draw[framecob] (0,4) -- (0,6) -- (1,6) -- (1,4) -- cycle;
	%top
	\draw[thick, teal] (-.5,2) node[below] { $\underline{x}$} -- (.25,4) ;
	\draw[thick, teal] (1.5,2) node[below] {$\underline{z}$} -- (.75,4);
	\draw[framecob] (-1,2) -- (0,4) -- (1,4) -- (2,2) -- (1,2) .. controls (1,3) and (0,3) .. (0,2)--cycle;
	\end{tikzpicture}.
	\end{aligned}
	\end{equation}
	The relations for $\lax(\mathsf{C})$ correspond to the following identities: 
	\begin{align}\label{eq:lax_diag_identities}
	%asso and unit			
	%\label{eq:lax_asso_unit}
	\begin{tikzpicture}[cob]
	\draw[framecob] (0,0) -- (1,2) -- (2,2) -- (3,0) -- (2,0) .. controls (2,1) and (1,1) ..  (1,0)--cycle;
	\draw[thick, teal] (.5,0)node[below] {$\underline{x}$} -- (1.33,2) -- (2.25,4) node[above left] {$\underline{x}$};
	\draw[thick, teal] (2.5,0)node[below] { $\underline{z}$} -- (1.66,2) -- (2.5,4) node[above] {$\underline{z}$};
	\draw[framecob] (1,2) -- (2,4) -- (3,4)--(5,0) -- (4,0) -- (3,2) .. controls (3,3) and (2,3) ..  %(2.5,3) -- 
	(2,2) -- cycle;
	\draw[thick, teal] (4.5,0) node[below] { $\underline{w}$} -- (2.75,4) node[above right] {$\underline{w}$};
	\end{tikzpicture}
	&=  \begin{tikzpicture}[cob,yscale=1,xscale=-1]
	\draw[framecob] (0,0) -- (1,2) -- (2,2) -- (3,0) -- (2,0) .. controls (2,1) and (1,1) ..  (1,0)--cycle;
	\draw[thick, teal] (.5,0)node[below] { $\underline{w}$} -- (1.33,2) -- (2.25,4) node[above right] { $\underline{w}$};
	\draw[thick, teal] (2.5,0)node[below] { $\underline{z}$} -- (1.66,2) -- (2.5,4) node[above] { $\underline{z}$};
	\draw[framecob] (1,2) -- (2,4) -- (3,4)--(5,0) -- (4,0) -- (3,2) .. controls (3,3) and (2,3) .. (2,2) -- cycle;
	\draw[thick, teal] (4.5,0) node[below] {$\underline{x}$} -- (2.75,4) node[above left] { $\underline{x}$};
	\end{tikzpicture},
	&&
	%Unit
	\begin{tikzpicture}[cob]
	\draw[framecob] (0,0) -- (0,2) -- (1,2) -- (1,0) -- cycle;
	\draw[thick, teal] (0.5,0) node[below]{ $\underline x$} -- (0.5,2) -- (1.5,4) node[above]{ $\underline x$};
	\draw[framecob] (0,2) -- (1,4) -- (2,4) -- (3,2) -- (2,2) .. controls (2,3) and (1,3) ..  (1,2)--cycle;
	\draw[framecob] (2,2) .. controls (2, 1) and (3,1) .. (3,2);
	\end{tikzpicture} \
	&=  \
	\begin{tikzpicture}[cob,xscale=-1]
	\draw[framecob] (0,0) -- (0,2) -- (1,2) -- (1,0) -- cycle;
	\draw[thick, teal] (0.5,0) node[below]{$\underline x$} -- (0.5,2) -- (1.5,4) node[above]{$\underline x$};
	\draw[framecob] (0,2) -- (1,4) -- (2,4) -- (3,2) -- (2,2) .. controls (2,3) and (1,3) ..  (1,2)--cycle;
	\draw[framecob] (2,2) .. controls (2, 1) and (3,1) .. (3,2);
	\end{tikzpicture}
	\ = \ 
	\begin{tikzpicture}[cob]
	\draw[thick,teal] (0,0) node[below]{$\underline x$} --  (0,2) node[above]{$\underline x$};
	\draw[framecob] (-.5,0) -- (.5,0) -- (.5,2) -- (-.5,2) -- cycle;
	\end{tikzpicture}\ .
	\end{align}
	%		\end{alignat}
	%	\end{subequations}
	%%%%%%%
	%oplax%
	%%%%%%%
	For $\oplax(\mathsf{C})$ we have instead the diagrams
	\begin{equation}
	\label{eq:oplaxdiag}
	\begin{aligned}
	\mathscr{k}_{x,z} &\mapsto \begin{tikzpicture}[cob,yscale=-1]
	\draw[thick, teal] (-.5,2) node[above] { $\underline{x}$} -- (.33,4) node[below] {$\underline{x}$};
	\draw[thick, teal] (1.5,2) node[above] {\footnotesize $\underline{z}$} -- (.66,4) node[below] {$\underline{z}$};
	\draw[framecob] (-1,2) -- (0,4) -- (1,4) -- (2,2) -- (1,2) .. controls (1,3) and (0,3) ..  (0,2)--cycle;
	\end{tikzpicture}, 
	&& 
	\mathscr q &\mapsto\ 
	\begin{tikzpicture}[cob]
	\draw[framecob] (0,0) .. controls (0,1) and (1,1) .. (1,0) -- cycle;
	\end{tikzpicture}\ , 
	\quad \text{with compatibility} \quad \begin{tikzpicture}[cob, yscale=-1]
	%first leg
	\draw[framecob] (-1,0) -- (-1,2) -- (0,2) -- (0,0) -- cycle;
	\draw[thick, teal] (-.5,0) node[above]{$\underline y$} -- (-.5,.75);
	\draw[thick, teal] (-.5,2) -- (-.5,1.25);
	\draw[thick, teal] (-.5,1) node{$f$} circle(.25);
	%second leg
	\draw[framecob] (2,0) -- (2,2) -- (1,2) -- (1,0) -- cycle;
	\draw[thick, teal] (2-.5,0) node[above]{$\underline w$} -- (2-.5,.75);
	\draw[thick, teal] (2-.5,2) -- (2-.5,1.25);
	\draw[thick, teal] (2-.5,1) node{$h$} circle(.25);
	%top
	\draw[thick, teal] (-.5,2) -- (.33,4) node[below] { $\underline{x}$};
	\draw[thick, teal] (1.5,2) -- (.66,4) node[below] { $\underline{z}$};
	\draw[framecob] (-1,2) -- (0,4) -- (1,4) -- (2,2) -- (1,2) .. controls (1,3) and (0,3) ..  (0,2)--cycle;
	\end{tikzpicture}\
	&  =
	\begin{tikzpicture}[cob, yscale=-1]
	%top
	\draw[thick, teal] (.25,4) -- (.25,4.8);
	\draw[thick, teal] (.25,6) node[below]{$\underline x$} -- (.25,5.2);
	\draw[thick, teal] (.75,4) -- (.75,4.8);
	\draw[thick, teal] (.75,6) node[below]{$\underline z$}-- (.75,5.2);
	\draw[thick, teal] (.25,5) node{$f$} circle(.2);
	\draw[thick, teal] (.75,5) node{$h$} circle(.2);
	\draw[framecob] (0,4) -- (0,6) -- (1,6) -- (1,4) -- cycle;
	%top
	\draw[thick, teal] (-.5,2) node[above] { $\underline{y}$} -- (.25,4) ;
	\draw[thick, teal] (1.5,2) node[above] {$\underline{w}$} -- (.75,4);
	\draw[framecob] (-1,2) -- (0,4) -- (1,4) -- (2,2) -- (1,2) .. controls (1,3) and (0,3) .. (0,2)--cycle;
	\end{tikzpicture},
	\end{aligned}
	\end{equation}
	and the relations %for $\oplax(\mathsf C)$
	correspond to the identities reversed from~\eqref{eq:lax_diag_identities}, reproduced below for convenience,
	
	\begin{align}
	\label{eq:oplax_diag_identities}
	%\label{eq:diag_oplax_functor}
	%coasso and counit
	%\label{eq:oplax_asso_unit}
	\begin{tikzpicture}[cob,yscale=-1]
	\draw[framecob] (0,0) -- (1,2) -- (2,2) -- (3,0) -- (2,0) .. controls (2,1) and (1,1) ..  (1,0)--cycle;
	\draw[thick, teal] (.5,0)node[above] {$\underline{x}$} -- (1.33,2) -- (2.25,4) node[below left] {$\underline{x}$};
	\draw[thick, teal] (2.5,0)node[above] { $\underline{z}$} -- (1.66,2) -- (2.5,4) node[below] {$\underline{z}$};
	\draw[framecob] (1,2) -- (2,4) -- (3,4)--(5,0) -- (4,0) -- (3,2) .. controls (3,3) and (2,3) ..  %(2.5,3) -- 
	(2,2) -- cycle;
	\draw[thick, teal] (4.5,0) node[above] { $\underline{w}$} -- (2.75,4) node[below right] {$\underline{w}$};
	\end{tikzpicture}
	&=  
	\begin{tikzpicture}[cob,yscale=-1,xscale=-1]
	\draw[framecob] (0,0) -- (1,2) -- (2,2) -- (3,0) -- (2,0) .. controls (2,1) and (1,1) ..  (1,0)--cycle;
	\draw[thick, teal] (.5,0)node[above] { $\underline{w}$} -- (1.33,2) -- (2.25,4) node[below right] { $\underline{w}$};
	\draw[thick, teal] (2.5,0)node[above] { $\underline{z}$} -- (1.66,2) -- (2.5,4) node[below] { $\underline{z}$};
	\draw[framecob] (1,2) -- (2,4) -- (3,4)--(5,0) -- (4,0) -- (3,2) .. controls (3,3) and (2,3) .. (2,2) -- cycle;
	\draw[thick, teal] (4.5,0) node[above] {$\underline{x}$} -- (2.75,4) node[below left] { $\underline{x}$};
	\end{tikzpicture},
	&
	%Counit
	\begin{tikzpicture}[cob,yscale=-1]
	\draw[framecob] (0,0) -- (0,2) -- (1,2) -- (1,0) -- cycle;
	\draw[thick, teal] (0.5,0) node[above]{ $\underline x$} -- (0.5,2) -- (1.5,4) node[below]{ $\underline x$};
	\draw[framecob] (0,2) -- (1,4) -- (2,4) -- (3,2) -- (2,2) .. controls (2,3) and (1,3) ..  (1,2)--cycle;
	\draw[framecob] (2,2) .. controls (2, 1) and (3,1) .. (3,2);
	\end{tikzpicture} \
	&=  \
	\begin{tikzpicture}[cob,xscale=-1,yscale=-1]
	\draw[framecob] (0,0) -- (0,2) -- (1,2) -- (1,0) -- cycle;
	\draw[thick, teal] (0.5,0) node[above]{$\underline x$} -- (0.5,2) -- (1.5,4) node[below]{$\underline x$};
	\draw[framecob] (0,2) -- (1,4) -- (2,4) -- (3,2) -- (2,2) .. controls (2,3) and (1,3) ..  (1,2)--cycle;
	\draw[framecob] (2,2) .. controls (2, 1) and (3,1) .. (3,2);
	\end{tikzpicture}
	\ = \ 
	\begin{tikzpicture}[cob]
	\draw[thick,teal] (0,0) node[below]{$\underline x$} --  (0,2) node[above]{$\underline x$};
	\draw[framecob] (-.5,0) -- (.5,0) -- (.5,2) -- (-.5,2) -- cycle;
	\end{tikzpicture}\ .
	\end{align}

	%%%%%%
	%Frob%
	%%%%%%
	
	Finally, for $\mathsf{F}(\mathsf{C})$ we have both \eqref{eq:laxdiag} and \eqref{eq:oplaxdiag} with relations~\eqref{eq:lax_diag_identities} and~\eqref{eq:oplax_diag_identities} augmented by the following Frobenius compatibility relations
	
	\begin{align}\label{eq:Frob_diag_identities}
	\begin{tikzpicture}[cob]
	\draw[framecob] (0,0) -- (1,0) .. controls (1,1) and (2,1).. (2,0) -- (3,0) -- (2,2) -- (3,4) -- (2,4) .. controls (2,3) and (1,3) .. (1,4) -- (0,4) -- (1,2) -- cycle;
	\draw[thick,teal] (0.33,0) node[below]{$\underline x$} .. controls (1.5,2) and (1.5,2) .. (0.5,4) node[above]{$\underline x$};
	\draw[thick,teal] (0.66,0) node[below]{$\underline z$} -- (2.33,4) node[above]{$\underline z$};
	\draw[thick,teal] (2.5,0) node[below]{$\underline w$} .. controls (1.5,2) and (1.5,2) .. (2.66,4) node[above]{$\underline w$};
	\end{tikzpicture}
	\quad &= \quad
	\begin{tikzpicture}[cob, xscale=1, yscale=-1]
	\draw[framecob] (0,0) -- (1,0) --(1,2) .. controls (1,3) and (2,3) .. (2,2) -- (3,0) -- (4,0) -- (5,2) -- (5,4) -- (4,4) -- (4,2) .. controls (4,1) and (3,1).. (3,2)-- (2,4) -- (1,4) -- (0,2) -- cycle;
	\draw[thick,teal] (1.33,4) node[below]{$\underline x$} .. controls (1.33,3) and (.5,3) .. (0.5,0) node[above]{$\underline x$};
	\draw[thick,teal] (1.66,4) node[below]{$\underline z$} -- (3.33,0) node[above]{$\underline z$};
	\draw[thick,teal] (4.5,4)node[below]{$\underline w$} .. controls (4.5,1) and (3.66,1) .. (3.66,0) node[above]{$\underline w$};
	\end{tikzpicture}\ ,  & \quad
	\begin{tikzpicture}[cob,xscale=-1]
	\draw[framecob] (0,0) -- (1,0) .. controls (1,1) and (2,1) .. (2,0) -- (3,0) -- (2,2) -- (3,4) -- (2,4) .. controls (2,3) and (1,3) ..  (1,4) -- (0,4) -- (1,2) -- cycle;
	\draw[thick,teal] (0.33,0) node[below]{$\underline w$} .. controls (1.5,2) and (1.5,2) .. (0.5,4) node[above]{$\underline w$};
	\draw[thick,teal] (0.66,0) node[below]{$\underline z$} -- (2.33,4) node[above]{$\underline z$};
	\draw[thick,teal] (2.5,0) node[below]{$\underline x$} .. controls (1.5,2) and (1.5,2) .. (2.66,4) node[above]{$\underline x$};
	\end{tikzpicture}
	\quad &= \quad
	\begin{tikzpicture}[cob,xscale=-1, yscale=-1]
	\draw[framecob] (0,0) -- (1,0) --(1,2) .. controls (1,3) and (2,3) .. (2,2) -- (3,0) -- (4,0) -- (5,2) -- (5,4) -- (4,4) -- (4,2) .. controls (4,1) and (3,1).. (3,2)-- (2,4) -- (1,4) -- (0,2) -- cycle;
	\draw[thick,teal] (1.33,4) node[below]{$\underline w$} .. controls (1.33,3.5) and (.5,2.5) .. (0.5,0) node[above]{$\underline w$};
	\draw[thick,teal] (1.66,4) node[below]{$\underline z$} -- (3.33,0) node[above]{$\underline z$};
	\draw[thick,teal] (4.5,4) node[below]{$\underline x$} .. controls (4.5,1) and (3.66,1) .. (3.66,0) node[above]{$\underline x$};
	\end{tikzpicture}\ .
	\end{align}

\begin{remark}
		Mulevi\v{c}ius considered a similar graphical calculus for monoidal functors, embedding string calculus in cylindrical ``tubes'' in the context of ribbon Frobenius functors~\cite[Fig.~4.1 (F1--F3)]{Mulevicius24}; see also Ponto--Schulman for lax symmetric monoidal functors~\cite[Fig.~7]{PS13}. Mulevi\v{c}ius' calculus also allows for braided and ribbon Frobenius functors. 
	\end{remark}
	
	\begin{lemma}\label{embeddingpreserving}
		Let $x$ be a right rigid object of $\mathsf{C}$ and let $x^{\ast}$ be a rigid right dual of $x$. Then $\underline{x^{\ast}}$ is a rigid right dual of $\underline{x}$ in $\frob(\mathsf{C})$. 
	\end{lemma}
	
	\begin{proof}
		Let $\eta: \bbone \rightarrow x^{\ast}\otimes x$ and $\varepsilon: x\otimes x^{\ast} \rightarrow \bbone$ be the unit and counit for the duality, which we denote in the string calculus of $\mathsf C$ and in the diagrammatics of $\frob (\mathsf C)$  as
		\begin{equation}
		\begin{aligned}
		\eta &= \begin{tikzpicture}[baseline=0,
		scale=.7,font=\scriptsize]
		\draw[thick, teal] (0.75,1) node[above]{$x$} .. controls (0.75,0.5) and (0.25,0.5) .. (0.25,1) node[above]{$x^*$};
		\draw[teal] (.5,0) node {$\bbone$}; 
		\end{tikzpicture}  &
		\varepsilon &= 	\begin{tikzpicture}[baseline=0,
		scale=.7,font=\scriptsize]
		\draw[thick, teal] (0.75,0) node[below]{$x^*$} .. controls (0.75,0.5) and (0.25,0.5) .. (0.25,0) node[below]{$x^{\phantom{*}}$};
		\draw[teal] (.5,1) node {$\bbone$};
		\end{tikzpicture}\ , &\qquad
		\eta &\mapsto 
		\begin{tikzpicture}[baseline=0, scale=.7,font=\scriptsize]
		\draw[framecob] (0,0) rectangle (1,1);
		\draw[thick, teal] (0.75,1) node[above]{$x$} .. controls (0.75,0.5) and (0.25,0.5) .. (0.25,1) node[above]{$x^*$};
		\end{tikzpicture}\ , &
		\varepsilon &\mapsto 
		\begin{tikzpicture}[baseline=0,scale=.7,font=\scriptsize]
		\draw[framecob] (0,0) rectangle (1,1);
		\draw[thick, teal] (0.75,0) node[below]{$x^*$} .. controls (0.75,0.5) and (0.25,0.5) .. (0.25,0) node[below]{$x^{\phantom{*}}$};
		\end{tikzpicture}\ .
		\end{aligned}
		\end{equation}
		Then the morphisms
		\begin{equation}\label{eq:frob_diag_rigid}
		\begin{aligned}
		\begin{tikzpicture}[cob]
		\draw[framecob] (0,0) -- (1,0) .. controls (1,1) and (2,1) .. (2,0) -- (3,0) .. controls (3,2) and (0,2) .. (0,0);
		\draw[thick, teal] (0.5,0) node[below] {$x^{\phantom{*}}$} .. controls (0.5,1.5) and (2.5,1.5) .. (2.5,0) node[below]{$x^*$};
		\end{tikzpicture}\ 
		&:=\
		\begin{tikzpicture}[cob]
		\draw[framecob] (0,0) -- (1,2) -- (2,2) -- (3,0) -- (2,0) .. controls (2,1) and (1,1) .. (1,0) -- cycle;
		\draw[framecob] (1,2) rectangle (2,3);
		\draw[framecob] (1,3) .. controls (1,4) and (2,4) .. (2,3);
		\draw[thick, teal] (.5,0) node[below]{$x^{\phantom{*}}$} .. controls (.5,1) and (1.25,1) .. (1.25,2);
		\draw[thick, teal] (2.5,0) node[below]{$x^*$} .. controls (2.5,1) and (1.75,1) .. (1.75,2);
		\draw[thick, teal] (1.75,2) .. controls (1.75,2.5) and (1.25,2.5) .. (1.25,2);
		\end{tikzpicture}, 
		&\quad 
		\begin{tikzpicture}[cob,yscale=-1]
		\draw[framecob] (0,0) -- (1,0) .. controls (1,1) and (2,1) .. (2,0) -- (3,0) .. controls (3,2) and (0,2) .. (0,0);
		\draw[thick, teal] (0.5,0) node[above] {$x^*$} .. controls (0.5,1.5) and (2.5,1.5) .. (2.5,0) node[above]{$x^{\phantom{*}}$};
		\end{tikzpicture}\ 
		&:= \
		\begin{tikzpicture}[cob,yscale=-1]
		\draw[framecob] (0,0) -- (1,2) -- (2,2) -- (3,0) -- (2,0) .. controls (2,1) and (1,1) ..  (1,0) -- cycle;
		\draw[framecob] (1,2) rectangle (2,3);
		\draw[framecob] (1,3) .. controls (1,4) and (2,4) .. (2,3);
		\draw[thick, teal] (.5,0)node[above]{$x^*$} .. controls (.5,1) and (1.25,1) .. (1.25,2) ;
		\draw[thick, teal] (2.5,0)node[above]{$x$} .. controls (2.5,1) and (1.75,1) .. (1.75,2);
		\draw[thick, teal] (1.75,2) .. controls (1.75,2.5) and (1.25,2.5) .. (1.25,2);
		\end{tikzpicture}
		\end{aligned}
		\end{equation}
		satisfy the zigzag equations 
		\begin{equation}
		\begin{aligned}
		\begin{tikzpicture}[cob]
		\draw[framecob] (0,0) -- (1,0) .. controls (1,1) and (2,1) .. (2,0) .. controls (2,-2) and (5,-2) .. (5,0) -- (4,0) .. controls (4,-1) and (3,-1).. (3,0) .. controls (3,2) and (0,2) .. (0,0);
		\draw[thick, teal] (0.5,0) node[below] {$x^{\phantom{*}}$} .. controls (0.5,1.5) and (2.5,1.5) .. (2.5,0) .. controls (2.5,-1.5) and (4.5,-1.5) .. (4.5,0) node[above]{$x^{\phantom{*}}$} ;
		\end{tikzpicture} \ &= \ \begin{tikzpicture}[cob]
		\draw[framecob] (0,0) rectangle (1,2);
		\draw[thick, teal] (0.5,0) node[below]{$x^{\phantom{*}}$} -- (.5,2) node[above]{$x^{\phantom{*}}$};
		\end{tikzpicture}\ , &\qquad 
		 \begin{tikzpicture}[cob]
			\draw[framecob] (0,0) rectangle (1,2);
			\draw[thick, teal] (0.5,0) node[below]{$x^{*}$} -- (.5,2) node[above]{$x^{*}$};
		\end{tikzpicture}
		 \ &= \
		 \begin{tikzpicture}[cob,xscale=-1]
		 \draw[framecob] (0,0) -- (1,0) .. controls (1,1) and (2,1) .. (2,0) .. controls (2,-2) and (5,-2) .. (5,0) -- (4,0) .. controls (4,-1) and (3,-1).. (3,0) .. controls (3,2) and (0,2) .. (0,0);
		 \draw[thick, teal] (0.5,0) node[below] {$x^{*}$} .. controls (0.5,1.5) and (2.5,1.5) .. (2.5,0) .. controls (2.5,-1.5) and (4.5,-1.5) .. (4.5,0) node[above]{$x^{*}$} ;
		 \end{tikzpicture}\ .
		\end{aligned}
		\end{equation}
	\end{proof}
	
	\begin{corollary}{\cite[Theorem~2]{DP08}}
		A Frobenius monoidal functor $\mathcal{F}: \mathsf{C}\rightarrow \mathsf{D}$ preserves dual pairs.
	\end{corollary}
	
	\begin{proof}
		Using the functor $\mathcal{E}: \mathsf{C} \rightarrow \frob(\mathsf{C})$ of \Cref{factorize}, we write $\mathcal{F} = \overline{\mathcal{F}} \circ \mathcal{E}$. Since $\overline{\mathcal{F}}$ is strict monoidal, it preserves dual pairs, and $\mathcal{E}$ preserves dual pairs by \Cref{embeddingpreserving}.
	\end{proof}

	As an example of application, we give the diagrammatic proof that, for Frobenius monoidal functors, a lax and oplax monoidal transformation between them is invertible. The diagrammatic proof was given in a talk by McCurdy~\cite{McCurdyTalk10}; see also~\cite[Prop.~7]{DP08}, \cite[Prop.~2.10]{PS13} for the statement.
	\begin{proposition}\label{prop:invertibleFrob}
		Given Frobenius monoidal functors $\funG, \funK:\catC \to \catD$ for $\catC$ rigid, if $\tau:\mathcal G\to \mathcal K$ is a lax and oplax monoidal transformation, then $\tau$ is invertible.
	\end{proposition}
	\begin{proof}
		We do a diagrammatic proof to construct the inverse of $\tau$. We draw the diagrams of the functor $\mathcal G$ as we did previously, and add colours and double lines to distinguish those of $\mathcal K$
		\begin{equation}
		\widehat{\mathcal G}(\ell_{x,z}) = \begin{tikzpicture}[cob]
		\draw[framecob] (-1,2) -- (0,4) -- (1,4) -- (2,2) -- (1,2) .. controls (1,3) and (0,3) ..  (0,2)--cycle;
		\draw[thick, teal] (-.5,2) node[below] { $\underline{x}$} -- (.33,4) node[above] {$\underline{x}$};
		\draw[thick, teal] (1.5,2) node[below] {$\underline{z}$} -- (.66,4) node[above] { $\underline{z}$};
		\end{tikzpicture}
		\qquad 
		\widehat{\mathcal K}(\ell_{x,z}) = \begin{tikzpicture}[cob]
		\draw[framecob,blue,double] (-1,2) -- (0,4) -- (1,4) -- (2,2) -- (1,2) .. controls (1,3) and (0,3) ..  (0,2)--cycle;
		\draw[thick, teal] (-.5,2) node[below] { $\underline{x}$} -- (.33,4) node[above] {$\underline{x}$};
		\draw[thick, teal] (1.5,2) node[below] {$\underline{z}$} -- (.66,4) node[above] {$\underline{z}$};
		\end{tikzpicture}
		\end{equation}
		We represent the natural transformation $\tau_x$ as a dotted line, and its naturality is expressed diagrammatically by allowing crossing:
		\begin{equation}
		\tau_x = 
		\begin{tikzpicture}[cob]
		\draw[framecob] (-1,0) -- (-1,-1)--(1,-1) -- (1,0);
		\draw[framecob, blue, double] (-1,0) -- (-1,1) -- (1,1) -- (1,0);
		\draw[thick, dashed] (-1.25,0) -- (1.25,0);
		\draw[thick, teal] (0,-1) node[below]{ $\underline x$}  -- (0,1)node[above]{$\underline x$};
		\end{tikzpicture}\ ,
		\qquad f\tau_x \to\ 
		\begin{tikzpicture}[cob]
		\draw[framecob] (-1,0) -- (-1,-1)--(1,-1) -- (1,0);
		\draw[framecob, blue, double] (-1,0) -- (-1,1) -- (1,1) -- (1,0);
		\draw[thick, dashed] (-1.25,0) -- (1.25,0);
		\draw[thick, teal] (0,0.5) node{ $f$} circle(.25);
		\draw[thick, teal] (0,-1) node[below]{$\underline x$}  -- (0,.25); \draw[thick, teal] (0,1-.25) -- (0,1)node[above]{$\underline y$};
		\end{tikzpicture} \ = \ 	
		\begin{tikzpicture}[cob]
		\draw[framecob] (-1,0) -- (-1,-1)--(1,-1) -- (1,0);
		\draw[framecob, blue, double] (-1,0) -- (-1,1) -- (1,1) -- (1,0);
		\draw[thick, dashed] (-1.25,0) -- (1.25,0);
		\draw[thick, teal] (0,-1) node[below]{ $\underline x$}  -- (0,-.75); \draw[thick, teal] (0,-.25) -- (0,1)node[above]{$\underline y$};
		\draw[thick, teal] (0,-.5) node{$f$} circle(.25);
		\end{tikzpicture} \leftarrow \tau_yf, \quad \text{for } f:x\to y.
		\end{equation}
		The fact that $\tau$ is lax monoidal expresses itself diagrammatically as
		\begin{equation}
		\begin{tikzpicture}[cob]
		\draw[framecob,blue,double] (-1.15,1-.3)-- (-.5,2) --(-.5,3) -- (.5,3) -- (.5,2) -- (1.15,1-.3); 
		\draw[framecob] (-1.15,1-.3) -- (-1.5,0) -- (-.5,0) .. controls (-.5,1) and (.5,1) .. (.5,0) -- (1.5,0) -- (1.15,1-.3);
		\draw[dashed] (-1.4,1-.3) -- (1.4,1-.3);
		\end{tikzpicture}
		\quad = \quad 	
		\begin{tikzpicture}[cob]
		\draw[framecob,blue,double] (-.5,2) --(-.5,3) -- (.5,3) -- (.5,2); 
		\draw[framecob] (-.5,2) -- (-1.5,0) -- (-.5,0) .. controls (-.5,1) and (.5,1) .. (.5,0) -- (1.5,0) -- (.5,2);
		\draw[dashed] (-1.1,2) -- (1.1,2);
		\end{tikzpicture} \ .
		\end{equation}
		We construct the inverse via the following diagrammatic construction using the rigidity of $\mathsf C$~\eqref{eq:frob_diag_rigid}
		\begin{equation}
		\tau_x^{-1} = \quad 
		\begin{tikzpicture}[cob]
		\draw[framecob] (-.5,0) .. controls (-.5,-.5) and (-1.5,-.5) .. (-1.5,0) --(-1.5,1)-- (-2.5,1) -- (-2.5,0) .. controls (-2.5,-1.5) and (.5,-1.5) .. (.5,0);
		%		\draw[very thick] (-1.5,1) -- (-1.5,0) .. controls (-1.5,-.5) and (-.5,-.5) .. (-.5,0);
		\draw[dashed] (-.75,0) -- (.75,0);
		\draw[framecob, double, blue] (-.5,0) .. controls (-.5,1.5) and (2.5,1.5) .. (2.5,0) -- (2.5,-1) -- (1.5,-1) -- (1.5,0) .. controls (1.5,.5) and (.5,.5) .. (.5,0);
		%		\draw[thick, double, blue] (.5,0) .. controls  (.5,.5) and (1.5,.5) .. (1.5,0) -- (1.5,-1);
		\draw[thick, teal] (-2,1) node[above]{$\underline x$} -- (-2,0) .. controls (-2,-1) and (0,-1) .. (0,0) .. controls (0,1) and (2,1) .. (2,0) -- (2,-1) node[below]{$\underline x$};
		\end{tikzpicture}\ .
		\end{equation}
	\end{proof}
	
	\section*{Acknowledgements}
	We thank Johannes Flake, Peter Kristel, Léo Schelstraete, Ivan Yakovlev, and Tony Zorman for discussion around this project. ALR acknowledges the hospitality of Uppsala University where this project started, and the support of FRQNT (ref.~326641) and from his Hausdorff postdoc, which is funded by the Deutsche Forschungsgemeinschaft under Germany's excellence Strategy -- GZ 2047/1, Projeckt-ID~390685813.
	MS\ is supported by the Knut and Alice Wallenberg Foundation (Grant No.~2024.0339). Large parts of this work were done while MS\ was supported by Lundström-Åmans stipendiestiftelse.
	
	%\section*{Declarations} %A: we might need to do it eventually
	%\begingroup
	%\singlespacing
	%\subsection*{Ethical Approval} Not applicable.
	%\subsection*{Competing interests} The authors declare no conflict of interest related to this work.
	%\subsection*{Authors' contribution} 
	%\subsection*{Funding}  ALR was supported by a postdoctoral scholarship of the FRQNT [id]
	%\subsection*{Availability of data and software}
	%Data sharing is not applicable to this article as no datasets
	%were generated or analysed during the study.
	%\endgroup
	\printbibliography
\end{document}